\documentclass[11pt]{amsart}

\usepackage{amssymb}
\usepackage{clipboard}
\usepackage[margin=30mm]{geometry}
\usepackage{mathtools}
\usepackage{microtype}
\usepackage{physics}
\usepackage{tikz}
	\usetikzlibrary{arrows.meta}
	\definecolor{green}{rgb}{0.0,0.5,0.0}

\usepackage[hidelinks]{hyperref}
\usepackage[capitalise,noabbrev]{cleveref}

\AtBeginDocument{%
   \def\MR#1{}
}

\newcommand{\RR}{\mathbb{R}}
\renewcommand{\SS}{\mathbb{S}}
\newcommand{\TT}{\mathbb{T}}
\newcommand{\ZZ}{\mathbb{Z}}
\newcommand{\calL}{\mathcal{L}}
\newcommand{\calP}{\mathcal{P}}
\newcommand{\vol}{\mathrm{Vol}}
\newcommand{\voct}{v_{\textup{oct}}}
\newcommand{\vtet}{v_{\textup{tet}}}
\newcommand{\svv}[1]{\mathbf{\widetilde{#1}}}
\newcommand{\vv}[1]{\mathbf{#1}}
\newcommand{\bdy}{\partial} 
\newcommand{\from}{\colon\thinspace}
\newcommand{\GL}{\mathrm{GL}}

\newtheorem{theorem}{Theorem}[section]
\newtheorem{proposition}[theorem]{Proposition}
\newtheorem{lemma}[theorem]{Lemma}
\newtheorem{corollary}[theorem]{Corollary}
\theoremstyle{definition}
\newtheorem{definition}[theorem]{Definition}
\newtheorem{remark}[theorem]{Remark}
\newtheorem{question}[theorem]{Question}

\newtheorem*{namedtheorem}{\theoremname}
\newcommand{\theoremname}{testing}
\newenvironment{named}[1]{\renewcommand{\theoremname}{#1}\begin{namedtheorem}}{\end{namedtheorem}}

\title[Volume bounds for hyperbolic rod complements in the $3$-torus]{Volume bounds for hyperbolic rod complements \\ in the $3$-torus}

\author{Norman Do}
\address{School of Mathematics, Monash University, VIC 3800, Australia}
\email{norm.do@monash.edu} 

\author{Connie On Yu Hui} 
\address{School of Mathematics, Monash University, VIC 3800, Australia} 
\email{onyu.hui@monash.edu | connieonyuhui@gmail.com} 

\author{Jessica S. Purcell}
\address{School of Mathematics, Monash University, VIC 3800, Australia}
\email{jessica.purcell@monash.edu}

\begin{document}

\begin{abstract} 
The study of rod complements is motivated by rod packing structures in crystallography. We view them as complements of links comprised of Euclidean geodesics in the $3$-torus. Recent work of the second author classifies when such rod complements admit hyperbolic structures, but their geometric properties are yet to be well understood. In this paper, we provide upper and lower bounds for the volumes of all hyperbolic rod complements in terms of rod parameters,  and show that these bounds may be loose in general. We introduce better and asymptotically sharp volume bounds for a family of rod complements. The bounds depend only on the lengths of the continued fractions formed from the rod parameters. 
\end{abstract}

\maketitle

\section{Introduction}

The present work is motivated by the notion of rod packing structures in crystallography. In 1977, O'Keeffe and Andersson observed that many crystal structures can be described as a packing of uniform cylinders, representing linear or zigzag chains of atoms or connected polyhedra~\cite{OKeeffe-Andersson:RodPCrystalChem}. In 2001, O'Keeffe et~al. classified some of the simplest so-called rod packings in terms of arrangements in Euclidean space~\cite{OKeeffeEtAl:CubicRodPackings}. Rod packings have also appeared in the biological science and materials science literature~\cite{RosiEtAl:RodPackingsMOF, ERH:PeriodicEntanglementII, Norlen-AA:CubicRodAndSkin, EH:SwollenCorneocytes}.

A rod packing structure exhibits translational symmetry along each dimension in a three-dimensional Euclidean space, it is thus natural to view a rod packing structure as a geodesic link in the $3$-torus, whose covering space is the three-dimensional Euclidean space. In this paper, we use tools from $3$-manifold geometry and topology to study the complements of these geodesic links, called \emph{rod complements}. In particular, Thurston's geometrisation theorem implies that each rod complement can be decomposed into geometric pieces. Indeed, each rod complement with three or more linearly independent rods is hyperbolic or has a unique hyperbolic rod complement component in its JSJ decomposition \cite{Hui-Purcell:RodPackings, Hui:ClassifyT3Rods, Burton-dePaiva-He-Hui:Crushing}. This means that geometric invariants such as the hyperbolic volume could be used to classify and distinguish rod complements. However, we need to know how geometry relates to the vector descriptions of rods in the $3$-torus, which aligns with the descriptions of rods in crystallography. This is still unknown in general.

In previous work, the second and third authors used vector description and the theory of links in the $3$-sphere to identify an infinite family of rod complements that admit complete hyperbolic structures \cite{Hui-Purcell:RodPackings}. Following this, the second author provided a complete classification of the geometric structures on rod complements in the $3$-torus~\cite{Hui:ClassifyT3Rods}. As a consequence, checking the hyperbolicity of a rod complement reduces to a linear algebra problem. 
While \cite{Hui-Purcell:RodPackings, Hui:ClassifyT3Rods} provide a convenient characterisation of when a rod complement is hyperbolic, they do not give further information on the metric. In this paper, we provide more information on the hyperbolic structures of rod complements via the study of their volumes.

The Mostow--Prasad rigidity theorem states that a complete hyperbolic metric on a finite-volume hyperbolic $3$-manifold is unique, so hyperbolic volume is a topological invariant. In the crystallographic setting, the uniqueness of hyperbolic structures allows us to associate each rod packing structure with a real number, namely the volume. When volumes are distinct, this provides a simple way to distinguish rod packing structures and avoid the more complicated symmetry descriptions that are often used in chemistry, see \cite{OKeeffeEtAl:CubicRodPackings} for examples.

For a rod complement in the $3$-torus, each rod has an associated direction in the unit cube fundamental region of the $3$-torus. We encode the direction of each rod by integer vector coordinates, which we call \emph{rod parameters}. Our most general result provides upper and lower volume bounds in terms of the number of rods and their rod parameters.

\begin{named}{\cref{Thm:VolBoundsDet}}
\Copy{VolBoundsDet}{Let $R_1, R_2, \ldots, R_n$ be disjoint rods in the 3-torus whose complement is a hyperbolic $3$-manifold $M$. After applying a linear homeomorphism and renumbering, if necessary, we may assume that there is a positive integer $k < n$ such that $R_{k+1}, R_{k+2}, \ldots, R_n$ are exactly the $(0,0,1)$-rods. Suppose that $R_i$ has direction vector $(p_i, q_i, z_i)$, for $i = 1, 2, \ldots, n$. Then we have the inequalities
\[
n \, \vtet < \vol(M) \leq 8 \, \vtet \left( \sum_{1 \leq i < j \leq k} \abs{p_i q_j - p_j q_i} + \sum_{1 \leq i \leq k} \left( \gcd(p_i,q_i) - 1 \right) \right),
\] 
where $\vtet \approx 1.01494$ is the volume of the regular ideal tetrahedron.}
\end{named}

The lower bound is due to a result proved by Adams, which applies to any cusped hyperbolic $3$-manifold~\cite{Adams:VolNCuspedMfds}. Such a bound can be loose in general; indeed, we find families of rod complements for which the number of rods is fixed at $n = 3$, but for which the volumes approach infinity.

The upper bound uses more recent results of Cremaschi and Rodr\'{\i}guez Migueles~\cite[Theorem~1.5]{Cremaschi-RodriguezMigueles:HypOfLinkCpmInSFSpaces}, which can be applied to many complements of geodesic links in Seifert fibred spaces, see also \cite{CRY:VolAndFillingCollectionsOfMulticurves}. Again, such a bound can be loose, even when restricted to rod complements; there are families of rod complements for which the volumes are bounded but for which the right side of the inequality above grows to infinity.

Thus, while \cref{Thm:VolBoundsDet} provides reasonable initial bounds that may be strong in certain cases, they are somewhat unsatisfying in general. It would be desirable to have upper and lower volume bounds that depend linearly on the same quantity. For example, hyperbolic volumes of 2-bridge knots~\cite{Gueritaud-AppByFuter:TwoBridgeLink}, alternating knots~\cite{Lackenby:AltVol}, and highly twisted knots~\cite{Futer-Kalfagianni-Purcell:DehnFillingVolumeJonesPolynomial} are known to be bounded above and below by linear functions of the number of twist regions. For all of these knot complements, the upper bound is asymptotically sharp. The lower bound is asymptotically sharp in the 2-bridge case~\cite{Gueritaud-AppByFuter:TwoBridgeLink}, and sharp, realised by the Borromean rings, in the alternating case~\cite[Theorem~2.2]{Agol-Storm-Thurston:LowerBounds}. Similarly, there are upper and lower volume bounds for adequate knots in terms of coefficients of coloured Jones polynomials~\cite{Dasbach-Lin:VolumishThmForJonesPolyOfAltKnots, Futer-Kalfagianni-Purcell:DehnFillingVolumeJonesPolynomial, FKP:Guts}. There are also upper and lower volume bounds for fibred 3-manifolds in terms of a quantity related to the action of the monodromy map~\cite{Brock:VolMappingTori}, with analogous results for cusp volumes~\cite{FuterSchleimer}. One would like to obtain such results for rod complements.

While we have not obtained coarse volume bounds of this form in general, we do find improved, asymptotically sharp volume bounds for infinite families of rod complements in terms of the lengths of the continued fractions formed from their rod parameters.  These lengths of continued fractions can remain the same when rod parameters increase significantly.

\begin{named}{\cref{Thm:MainOrthogonalRods}}
\Copy{MainOrthogonalRods}{Let $R_1, R_2, \ldots, R_n$ be disjoint rods in the $3$-torus whose complement is $M$, where $n \geq 3$. Suppose that $R_n$ has direction vector $(0,0,1)$ and for $i < n$, $R_i$ has direction vector $(p_i, q_i, 0)$, with $(p_i, q_i) \neq (p_{i+1}, q_{i+1})$ for $i = 1, 2, \ldots, n-2$ and $(p_{n-1}, q_{n-1}) \neq (p_1,q_1)$. Suppose that $R_1, R_2, \ldots, R_{n-1}$ are positioned from top to bottom in the unit cube representation of the $3$-torus. Let $[c_{i1}; c_{i2}, \ldots, c_{im_i}]$ be a continued fraction expansion for $p_i/q_i$. Then $M$ is hyperbolic and its volume satisfies the asymptotically sharp upper bound
\[
\vol(M) \leq 2 \, \voct \sum_{i=1}^{n-1} m_i.
\]
Suppose in addition that 
\[
C \coloneq \min_{\substack{1 \leq i \leq n-1 \\ j \geq 2}} \{|c_{ij}|, |c_{i1}-c_{(i-1) 1}|\} \geq 6,
\]
where $c_{01}$ is interpreted as $c_{(n-1)1}$. Then the volume satisfies the lower bound
\[
\vol(M) \geq \left( 1 - \frac{4\pi^2}{C^2+4} \right)^{3/2} 2 \, \voct \sum_{i=1}^{n-1} m_i.
\]}
\end{named}

\cref{Thm:MainOrthogonalRods} leads to the following consequences.

\begin{named}{\cref{Cor:BadUpperBound}}
\Copy{BadUpperBound}{There exists a sequence of hyperbolic rod complements with bounded volume, but for which the upper bound of \cref{Thm:VolBoundsDet} grows to infinity.}
\end{named}

\begin{named}{\cref{Cor:3CuspedInfVol}} 
\Copy{3CuspedInfVol}{There exists a sequence of hyperbolic rod complements, each with three rods, whose volumes grow to infinity.}
\end{named}

Other works related to geometry and periodic links include \cite{ERH:IdealGeometry}, in which Evans, Robins, and Hyde studied 3-periodic links using energy functions. In~\cite{ES:HyperbolicBiological}, Evans and Schr\"oder-Turk use 2-dimensional hyperbolic geometry to study triply periodic links embedded in the $3$-dimensional Euclidean space.

The structure of the paper is as follows.
\begin{itemize}
\item In \cref{Sec:Prelim}, we introduce some terminology, notation and foundational results that are used throughout the paper. These pertain to rod complements, continued fractions and homeomorphisms from the $n$-dimensional torus to itself. 
\item In \cref{Sec:VolBounds}, we provide general volume bounds for all hyperbolic rod complements in the 3-torus (\cref{Thm:VolBoundsDet}). The upper bound is in terms of the rod parameters, while the lower bound is only in terms of the number of rods.
\item In \cref{Sec:Nested}, we introduce the notion of nested annular Dehn filling in the 3-torus.
\item In \cref{Sec:3CuspedInfVol}, we use the notion of nested annular Dehn filling to provide more refined volume bounds for a particular class of rod complements (\cref{Thm:MainOrthogonalRods}). This is sufficient to exhibit a family of rod complements with bounded volumes for which the upper bound of \cref{Thm:VolBoundsDet} grows to infinity (\cref{Cor:BadUpperBound}) and another family with bounded number of rods whose volumes grow to infinity (\cref{Cor:3CuspedInfVol}).
\item In \cref{Sec:Discussion}, we conclude with brief discussion of open questions that are motivated by the present work.
\end{itemize}

\subsection*{Acknowledgements}

We thank Jos\'{e} Andr\'{e}s Rodr\'{\i}guez-Migueles for helpful conversations. This work was partially supported by the Australian Research Council grant DP240102350. The first author was supported by the Australian Research Council grant FT240100795. The second author was supported by Monash International Tuition Scholarships and Monash Graduate Scholarship. We thank the anonymous reviewer for the helpful suggestions and comments.

\section{Preliminaries} \label{Sec:Prelim} 

\subsection{Rod complements}

We consider the $3$-torus $\TT^3$ as the unit cube $[0,1] \times [0,1] \times [0,1]$ in $3$-dimensional Euclidean space, with opposite faces glued identically, as in \cite{Hui-Purcell:RodPackings, Hui:ClassifyT3Rods}. Its universal cover is $\RR^3$ and it inherits the Euclidean metric from $\RR^3$.

A \emph{rod} is the projection of a Euclidean straight line with rational slope in $\RR^3$ to $\TT^3$ under the covering map.

For $n$ a positive integer, an \emph{$n$-rod complement} is the complement of $n$ disjoint rods in the $3$-torus. When $n$ is unspecified, we refer to such a manifold simply as a \emph{rod complement}.

Let $p, q, z$ be integers, not all zero, with $\gcd(p,q,z) = 1$. A $(p,q,z)$-rod is a geodesic in $\TT^3$ that has $(p,q,z)$ as a tangent vector. We call $(p,q,z)$ a \emph{direction vector} of the rod, where we consider $(p,q,z)$ only up to a change of sign. We call the integers $p, q, z$ the \emph{rod parameters} of the rod. A \emph{standard rod} is a $(1,0,0)$-rod, a $(0,1,0)$-rod, or a $(0,0,1)$-rod.

A rod complement is said to be \emph{hyperbolic} if it admits a complete hyperbolic structure; for further details on hyperbolic geometry, see for example~\cite{Purcell:HyperbolicKnotTheory}. In previous work, the second author classified exactly when rod complements are hyperbolic, Seifert fibred or toroidal.

\begin{theorem}[Hui,~\cite{Hui:ClassifyT3Rods}] \label{Thm:ClassifyingT3Rods}
Let $R_1, R_2, \ldots, R_n$ be disjoint rods in $\TT^3$. The rod complement $\TT^3 \setminus (R_1 \cup R_2 \cup \cdots \cup R_n)$ is:
\begin{enumerate}
\item \label{condition:hypIff} hyperbolic if and only if $\{R_1, R_2, \ldots, R_n\}$ contains three linearly independent rods and each pair of disjoint parallel rods are not linearly isotopic in the complement of the other rods;
\item \label{condition:SFIff} Seifert fibred if and only if all rods have the same direction vector; and 
\item \label{condition:ToroidalIf} toroidal if 
\begin{enumerate}
\item the direction vectors of the rods all lie in the same plane; or 
\item there exist two distinct rods that are linearly isotopic in the complement of the other rods.
\end{enumerate} 
\end{enumerate} 
\end{theorem}

In case (3)(b), suppose without loss of generality that $R_{n-1}$ and $R_n$ are linearly isotopic in the complement of the other rods. Then an essential torus encircling the linearly isotopic rods will cut the rod complement into a solid torus containing $R_{n-1}$ and $R_n$, and a new rod complement with rods $R_1, R_2, \ldots, R_{n-1}$. So if there were three linearly independent rods to begin with, there would be a unique hyperbolic rod complement appearing as a component of the JSJ decomposition; see~\cite[Theorem~19]{Burton-dePaiva-He-Hui:Crushing}. The upshot of this discussion is that rod complements are very commonly hyperbolic, in a certain sense.

Observe that in a hyperbolic rod complement, there may be several rods with the same direction vector, provided that for any two such rods, at least one other rod intersects the linear annuli bound by them. Two or more rods with the same direction vector are said to be \emph{parallel}.

\subsection{Continued fractions} \label{subsec:continued}

Let $p$, $q$ be nonzero relatively prime integers. Without loss of generality, we may assume $q > 0$ unless otherwise specified. The rational number $p/q$ can be expressed as a finite continued fraction
\[
\frac{p}{q} = [c_1;c_2,\ldots,c_m] \coloneq c_1+\cfrac{1}{c_2+\cfrac{1}{c_3 + \cfrac{1}{\ddots + \cfrac{1}{c_m}}}} \, ,
\]
where $c_1$ is an integer and $c_2, \ldots, c_m$ are non-zero integers. The integers $c_1, c_2, \ldots, c_m$ are called \emph{coefficients} or \emph{terms} of the continued fraction and the number $m$ is called the \emph{length} of the continued fraction.

Observe that a continued fraction expansion for a given rational number is not unique. For example, the rational number $\frac{7}{4}$ can be expressed in several ways, including $[1; 1, 3]$, $[1; 1, 2, 1]$ and $[2; -4]$. The upper bound of \cref{Thm:MainOrthogonalRods} is strengthened by using continued fraction expansions that have minimal length. In particular, if $m \geq 2$, we do not allow $c_m = 1$ in the continued fraction expansion above.

Note that the length of the continued fraction $[0] = \frac{0}{1}$ is one. For convenience, we define the ``empty'' continued fraction $[~] = \frac{1}{0}$ and consider its length to be zero.

The rational numbers whose continued fraction expansions we consider arise as slopes on the 2-torus. We consider the $2$-torus $\TT^2$ as the unit square $[0,1] \times [0,1]$ in $2$-dimensional Euclidean space, with opposite faces glued identically. Its universal cover is $\RR^2$ and it inherits the Euclidean metric from $\RR^2$.

Let $p$ and $q$ be integers, not both zero, with $\gcd(p, q) = 1$. A simple closed geodesic on $\TT^2$ is said to have \emph{slope $p/q$} or to be a \emph{$(p,q)$-curve} if it is isotopic to the projection of a line in $\RR^2$ with slope $\frac{q}{p}$. Observe that our definition of slope on the torus is the reciprocal of the corresponding slope on the plane. We defined slope of simple closed geodesics in this way because of our choices of notations in Section~\ref{Sec:Nested}.

\subsection{Homeomorphisms of the \texorpdfstring{$n$}{n}-torus}

The following are useful results concerning homeomorphisms of the $n$-dimensional torus $\TT^n$. The statements are well-known, but short proofs have been provided for completeness.

\begin{lemma} \label{Lem:GLnZ}
For $n\geq 2$, an element $A \in \GL(n, \ZZ)$ induces a homeomorphism from $\TT^n$ to itself.
\end{lemma}

\begin{proof}
The element $A \in \GL(n,\ZZ)$ gives rise to a homeomorphism from $\RR^n$ to itself that sends the integer lattice $\ZZ^n$ to itself. In particular, it takes the standard basis of $\RR^n$ to a basis formed by the columns of $A$, whose coordinates are integers. This produces a new fundamental domain for the torus. The induced homeomorphism simply maps the standard fundamental domain of the torus to this new fundamental domain via $A$. 
\end{proof}

In fact, it is known that when $n=2$ or $n=3$, $\GL(n, \ZZ)$ is the mapping class group of $\TT^n$. (The result for $n=2$ appears in~\cite[Theorem~2.5]{Farb-Margalit:APrimerOnMCG} while the result for $n=3$ follows from work of Hatcher~\cite{Hatcher:HomeoLarge3Mfds}.)

\begin{remark}
Given a rod complement in the $3$-torus that contains an $(a,b,c)$-rod $R$, there exists an element of $\GL(3, \ZZ)$ that sends $(a, b, c)$ to $(0, 0, 1)$. By \cref{Lem:GLnZ}, we may change the fundamental region of the \mbox{$3$-torus} to ensure that $R$ is a $(0,0,1)$-rod. In the rest of the paper, we often assume without loss of generality that one of the rods in a rod complement has direction vector $(0,0,1)$.
\end{remark}

\begin{lemma} [B\'ezout's lemma in $n$-dimensions] \label{Lem:gcdAndDet}
Let $n\geq 2$ be an integer. Suppose that $\vv{{a_n}} = (a_{1n}, a_{2n}, \ldots, a_{nn})^\intercal$ is a nonzero vector in $\ZZ^n \subset \RR^n$ with $\gcd(a_{1n}, a_{2n}, \ldots, a_{nn}) = 1$. Then there exist vectors $\vv{a}_1, \vv{a}_2, \ldots, \vv{a}_{n-1}$ in $\ZZ^n$ such that $\det(\vv{a}_1, \vv{a}_2, \ldots, \vv{a}_n) = 1$.
\end{lemma}

\begin{proof}
We prove the result by induction on $n$. Suppose that $\vv{a}_2 = (a_{12}, a_{22})^\intercal$ is a nonzero vector in $\ZZ^2$ with $\gcd(a_{12},a_{22}) = 1$. By B\'{e}zout's lemma, there exist integers $a_{11}$, $a_{21}$ such that $a_{11}a_{22}-a_{21}a_{12} = 1$. So defining $\vv{a}_1 = (a_{11}, a_{21})^\intercal$ leads to $\det(\vv{a}_1, \vv{a}_2) = 1$. This proves the base case $n = 2$.

Now let $n \geq 3$ be an integer. Suppose that $\vv{a}_n = (a_{1n}, a_{2n}, \ldots, a_{nn})^\intercal$ is a nonzero vector in $\ZZ^n$ with $\gcd(a_{1n}, a_{2n}, \ldots, a_{nn}) = 1$. Without loss of generality, suppose that $a_{nn} \neq 0$ so that the vector $\svv{a}_n \coloneq (a_{2n}, a_{3n}, \ldots, a_{nn})^\intercal$ is nonzero. Let 
\[
d \coloneq \gcd \left( a_{2n}, a_{3n}, \ldots, a_{nn} \right).
\]
Since $\gcd(a_{1n}, d) = \gcd(a_{1n}, a_{2n}, \ldots, a_{nn}) = 1$, by B\'{e}zout's lemma, there exist integers $s$ and $t$ such that $s d - t a_{1n} = 1$.

Set $a_{11} = s$ and 
\[
(a_{21}, a_{31}, \ldots, a_{n1}) \coloneq \frac{t}{d} \, \svv{a}_n^{\,\intercal} = \frac{t}{d} \left( a_{2n}, a_{3n}, \ldots, a_{nn} \right).
\]
Since $\frac{1}{d} \, \svv{a}_n \in \ZZ^{n-1}$ and $\gcd \left( \frac{a_{2n}}{d}, \frac{a_{3n}}{d}, \ldots, \frac{a_{nn}}{d} \right) = 1$, by induction there exist $\svv{a}_2, \svv{a}_3, \ldots, \svv{a}_{n-1}$ in $\ZZ^{n-1}$ such that $\det(\svv{a}_2, \svv{a}_3, \ldots, \svv{a}_{n-1}, \frac{1}{d} \, \svv{a}_n) = 1$.

Now define $\vv{a}_1 \coloneq (s, \frac{t}{d} \, \svv{a}_n^{\,\intercal})^\intercal, \vv{a}_2 \coloneq (0, \svv{a}_2^{\,\intercal})^\intercal, \ldots, \vv{a}_{n-1} \coloneq (0, \svv{a}_{n-1}^{\,\intercal})^\intercal$. Then by expanding along the first row, we find that
\begin{align*}
& \det( \vv{a}_1, \vv{a}_2, \ldots, \vv{a}_{n-1}, \vv{a}_n) \\
={}& a_{11} \det(\svv{a}_2, \ldots, \svv{a}_{n-1}, \svv{a}_{n}) + (-1)^{1+n} a_{1n} \det (\tfrac{t}{d} \, \svv{a}_n, \svv{a}_2, \ldots, \svv{a}_{n-1}) \\
={}& s d \, \det(\svv{a}_2, \ldots, \svv{a}_{n-1}, \tfrac{1}{d} \, \svv{a}_n) + (-1)^{(1+n)+(n-2)}a_{1n} t \, \det(\svv{a}_2, \ldots, \svv{a}_{n-1}, \tfrac{1}{d} \, \svv{a}_n) \\
={}& sd - a_{1n} t \\
={}& 1.
\end{align*}
This concludes the induction.
\end{proof}

\begin{proposition} \label{Prop:OneRodHomeo}
For fixed $n \geq 2$, all 1-rod complements in the $n$-torus are homeomorphic.
\end{proposition} 

\begin{proof}
Let $R$ be a rod in the $n$-torus whose fundamental region is $[0,1]^n$. Suppose $\vv{a}_n = (a_{1n}, a_{2n}, \ldots, a_{nn})^\intercal$ is the direction vector of $R$. We may translate the rod $R$ so that it intersects the origin. As $R$ is a simple closed curve, we must have $\gcd(a_{1n}, a_{2n}, \ldots, a_{nn}) = 1$. By \cref{Lem:gcdAndDet}, there exist vectors $\vv{a}_1, \vv{a}_2, \ldots, \vv{a}_{n-1}$ in $\ZZ^n\subset \RR^n$ such that 
\[
\det(\vv{a}_1, \vv{a}_2, \ldots, \vv{a}_n) = 1.
\] 
Hence, the matrix $(\vv{a}_1, \vv{a}_2, \ldots, \vv{a}_n)$ lies in $\GL(n,\ZZ)$ and by \cref{Lem:GLnZ}, it induces a homeomorphism that maps the $(0, 0, \ldots, 0,1)$-rod to the $\vv{a}_n$-rod. Therefore, any 1-rod complement $\TT^n \setminus R$ is homeomorphic to $\TT^n \setminus R_z$, where $R_z$ represents a standard $(0, 0, \ldots, 0, 1)$-rod.
\end{proof}

\section{Volume bounds for all rod complements} \label{Sec:VolBounds} 

In this section, we obtain upper and lower bounds on the volumes of all hyperbolic rod complements.

\begin{proposition} \label{Prop:AsSFSpace}
An $n$-rod complement in the $3$-torus with $k \geq 1$ parallel rods is an $(n-k)$-rod complement in the Seifert fibred space $\TT^2_{k} \times \SS^1$, where $\TT^2_{k}$ is a torus with $k$ punctures. 
\end{proposition}

\begin{proof}
Let $M$ be an $n$-rod complement in the $3$-torus with $k$ parallel rods $R_1, R_2, \ldots, R_k$. Suppose that these parallel rods have direction vector $(a, b, c)$, where $a, b, c$ are integers such that $\gcd(a, b, c) = 1$. By \cref{Lem:gcdAndDet}, there exist integers $f, g, h, p, q, r$ such that 
\[
\det\begin{pmatrix}
a & f & p\\ 
b & g & q\\ 
c & h & r 
\end{pmatrix} = 1
\qquad \Rightarrow \qquad
\begin{pmatrix}
a & f & p\\ 
b & g & q\\ 
c & h & r 
\end{pmatrix} \in \GL(3,\ZZ).
\]
By \cref{Lem:GLnZ}, such a matrix represents an orientation-preserving homeomorphism of $\TT^3$ sending the rods with direction vectors $(1,0,0)$, $(0,1,0)$, $(0,0,1)$ to rods with direction vectors $(a,b,c)$, $(f,g,h)$, $(p,q,r)$, respectively.

Define $T \subset \TT^3$ to be a 2-torus spanned by the vectors $(f,g,h)$ and $(p,q,r)$. Note that $T\setminus (R_1\cup R_2\cup \cdots \cup R_k)$ is a $k$-punctured torus. As the homeomorphism represented by the above matrix sends the standard fundamental region of the $3$-torus to the fundamental region spanned by the vectors $(a,b,c)$, $(f,g,h)$, $(p,q,r)$, $M$ is homeomorphic to an $(n-k)$-rod complement in the Seifert fibred space $T \setminus (R_1\cup R_2 \cup \cdots \cup R_k) \times \SS^1$. 
\end{proof}

\begin{theorem} \label{Thm:VolBoundsDet}
\Paste{VolBoundsDet}
\end{theorem}

\begin{proof}
From \cref{Thm:ClassifyingT3Rods}, we deduce that $n \geq 3$. Adams proved that an $n$-cusped hyperbolic 3-manifold $M$ with $n \geq 3$ satisfies the inequality $\vol(M) > n \, \vtet$, which is the desired lower bound~\cite[Theorem~3.4]{Adams:VolNCuspedMfds}.

We obtain the upper bound using a result of Cremaschi and Rodr\'{\i}guez-Migueles~\cite[Theorem~1.5]{Cremaschi-RodriguezMigueles:HypOfLinkCpmInSFSpaces}. They prove that for a link $\overline{\calL}$ in an orientable Seifert fibred space $N$ over a hyperbolic 2-orbifold~$O$ in which $\calL$ projects injectively to a filling geodesic multi-curve $\calL \subseteq O$, one has the volume bound
\[
\vol(N \setminus \overline{\calL}) < 8 \, \vtet \, i(\calL, \calL).
\]
Here, $i(\calL, \calL)$ denotes the geometric self-intersection number of $\calL$.

In our particular setting, \cref{Prop:AsSFSpace} asserts that $M$ is homeomorphic to a $k$-rod complement in the Seifert fibred space 
\[
N = \TT^2_{n-k} \times \SS^1,
\] 
where $\TT^2_{n-k} \coloneq \left( T \setminus (R_{k+1} \cup R_{k+2} \cup \cdots \cup R_n) \right)$ with $T \subseteq \TT^3$ being the $2$-torus with fundamental region equal to the unit square $[0,1]^2$ in the $xy$-plane. Observe that $T$ is a 2-torus such that the intersection number between $R_n$ and $T$ is 1. Denote by~$\overline{\calL}$ the $k$-component link $R_1 \cup R_2 \cup \cdots \cup R_k$ in $N$. Here, $R_i$ is a $(p_i, q_i, z_i)$-rod with $(p_i, q_i) \neq (0, 0)$ for $i = 1, 2, \ldots, k$.

Let $\mathcal{P} \from N \to \TT^2_{n-k}$ be the bundle projection map. Note that the link~$\overline{\calL}$ projects to $\calL$, a union of $k$ rods in the base space $\TT^2_{n-k}$, which is a 2-torus in~$\TT^3$ with $n-k$ punctures. The rod $R_i$ projects to a $(p_i,q_i)$-curve on this punctured torus.

After a small deformation of the rods, we may ensure that their projections intersect transversely, with at most two arcs meeting at each intersection point. Since the number of rods is finite, we can also ensure, up to small deformation, that 
\begin{enumerate}
    \item any pair of projections $\calP(R_i)$ and $\calP(R_j)$ intersect exactly $\abs{p_iq_j-p_jq_i}$ times; see for example~\cite[Section~1.2.3]{Farb-Margalit:APrimerOnMCG}; and
    \item the $(p_i,q_i)$-curve $\calP(R_i)$ intersects itself exactly $\gcd(p_i, q_i) - 1$ times. 
\end{enumerate}

Item (2) above can be seen as follows. Consider a $(p_i,q_i)$-curve $\gamma$ with $d_i\coloneqq\gcd(p_i, q_i)>1$. Up to homeomorphism of the torus, $\gamma$ is equivalent to the $(d_i,0)$-curve, for which a representative consists of $d_i-1$ horizontal arcs connected by a single arc running across, meeting $(d_i-1)$ strands. 


Hence, the total geometric intersection number of $\calP(\calL)$ is 
\[
\sum_{1 \leq i < j \leq k} \abs{p_i q_j - p_j q_i} + \sum_{1 \leq i \leq k} \left( \gcd(p_i,q_i) - 1 \right).
\]
Thus, applying the result of Cremaschi and Rodr\'{\i}guez-Migueles leads to the upper bound.
\end{proof}

\begin{remark}
The upper volume bound in \cref{Thm:VolBoundsDet} depends on the choice of rod that is sent to the $(0,0,1)$-rod via a homeomorphism of $\TT^3$. For example, if we consider four rods $R_1, R_2, R_3, R_4$ with direction vectors $(2,4,3), (5,7,1), (9,8,6), (0,0,1)$, respectively, \cref{Thm:VolBoundsDet} will give us an upper volume bound $8 \, \vtet \times 50$. Using the constructive proof of \cref{Lem:gcdAndDet}, we obtain the following matrices in $\GL(3,\ZZ)$ that map $(0,0,1)$ to $R_1, R_2, R_3$, respectively.
\[
\begin{pmatrix}
1 & 0 & 2\\
0 & -1 & 4\\
0 & -1 & 3
\end{pmatrix} \qquad 
\begin{pmatrix}
1 & 0 & 5\\
0 & 1 & 7\\
0 & 0 & 1
\end{pmatrix} \qquad
\begin{pmatrix}
-4 & 0 & 9\\
-4 & -1 & 8\\
-3 & -1 & 6
\end{pmatrix}
\]
By taking the inverses of these matrices and computing the new rod parameters, we now obtain upper volume bounds of $8 \, \vtet \times 116$, $8 \, \vtet \times 114$, and $8 \, \vtet \times 132$, respectively. The minimum among all such choices naturally provides an upper bound. 
\end{remark}

\section{Nested annular Dehn filling in the 3-torus} \label{Sec:Nested} 

We will show that neither the upper nor lower bound of \cref{Thm:VolBoundsDet} can be part of a two-sided coarse volume bound in terms of the given parameters. That is, we exhibit a family of rod complements with fixed number of cusps whose volumes grow to infinity as well as a family of rod complements with bounded volume for which the intersection number in the upper bound of \cref{Thm:VolBoundsDet} grows to infinity. For both of these results, we use the machinery of annular Dehn filling.

\begin{figure} 
\includegraphics[scale=0.75]{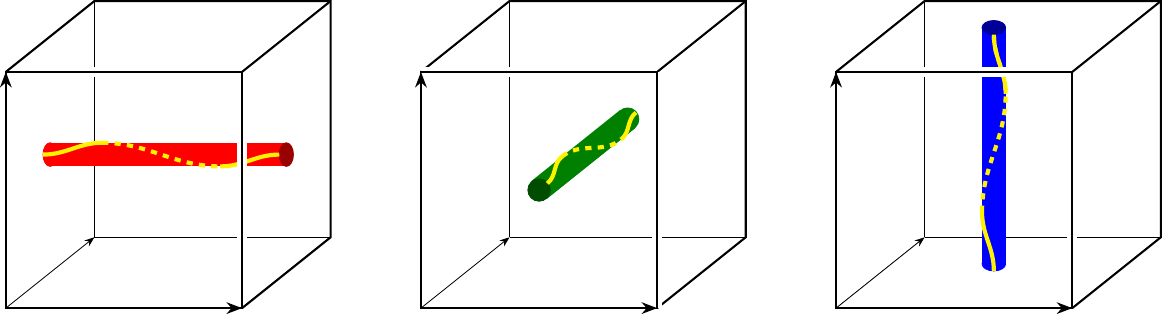}
\caption{A $(1,1)$-curve on the standard rods with direction vectors $(1,0,0)$, $(0,1,0)$ and $(0,0,1)$, respectively.}
\label{Fig:100Rod11curve} 
\end{figure}

\begin{definition} [Annular Dehn filling]
Let $A$ be an annulus embedded in a 3-manifold~$M$, with boundary curves $L^+$ and $L^-$. Let $\mu^{\pm}$ denote a meridian of $\bdy N(L^{\pm})$ and let $\lambda^{\pm}$ denote a longitude of $\bdy N(L^{\pm})$ that is parallel to $\bdy A$. 

For an integer $n$, define \emph{$(1/n)$-annular Dehn surgery} to be the process of drilling $N(L^+)$ and $N(L^-)$ from $M$, performing $(+1/n)$-Dehn filling on $\bdy N(L^+)$ and performing $(-1/n)$-Dehn filling on $\bdy N(L^-)$.

The surgery can be realised by cutting along $A$, performing $n$ Dehn twists along the core of $A$ in the anticlockwise direction (where the induced orientation puts $L^+$ on the right of the core of $A$), and then regluing; see for example~\cite[Section~2.3]{Baker}.

If the curves $L^{\pm}$ are already drilled, such as in the case of a link complement, define \emph{$(1/n)$-annular Dehn filling along $A$} to be the process of performing $(+1/n)$-Dehn filling on $L^+$ and performing $(-1/n)$-Dehn filling on $L^-$, where the framing on the link components is as above.
\end{definition}

In our case, we perform annular Dehn filling on an annulus bounded by a pair of parallel rods in the $3$-torus. Note that parallel rods bound many annuli in $\TT^3$. The following result confirms that the resulting link is well defined, regardless of our choice of annulus.

\begin{lemma} \label{Lem:AnnularDFIndependentA}
Let $R^+$ and $R^-$ be parallel rods in $\TT^3$ that form the boundary of two non-isotopic annuli $A^+$ and $A^-$ with disjoint interiors. Suppose that $A^+$ is the annulus oriented with $R^+$ on the right of the core, under the induced orientation from $\TT^3$. Then $(1/n)$-annular Dehn filling on $A^+$ and $(-1/n)$-annular Dehn filling on $A^-$ result in homeomorphic manifolds.

More generally, suppose that $A_1$ and $A_2$ are disjoint annuli with $A_1$ cobounded by rods $R_0$ and $R_1$, with $R_1$ to the right, and $A_2$ cobounded by $R_1$ and a rod $R_2$, with $R_1$ to the left. Let $M$ be the result of performing $(1/n)$-annular Dehn filling on $A_1$ followed by $(1/m)$-annular Dehn filling on $A_2$. Then $M$ is also the result of performing $(-1/n)$-Dehn filling on $R_0$, followed by $(1/(n-m))$-Dehn filling on $R_1$, followed by $(1/m)$-Dehn filling on $R_2$, when $R_0 \neq R_2$. If $R_0=R_2$, then the Dehn filling coefficient on $R_0=R_2$ is $1/(m-n)$.
\end{lemma}

\begin{proof}
Let $N^+$ be the manifold obtained by $1/n$-annular Dehn filling $A^+$ and let $N^-$ be the manifold obtained by $-1/n$-annular Dehn filling $A^-$. The fact that $N^+$ and $N^-$ are homeomorphic follows from the fact that the link complements have the same Dehn surgery coefficients. Thus, the results of the Dehn fillings must be homeomorphic.

To prove the more general statement, we again consider the Dehn surgery coefficients. Annular Dehn filling first along $A_1$ gives surgery slope $\mu + n\lambda$ on $R_1$ and $\mu - n \lambda$ on $R_0$, where $\mu$ denotes a meridian and $\lambda$ is parallel to $\bdy A_1$. Then performing $1/m$-annular Dehn filling along $A_2$ adjusts the surgery slope on $R_1$ by subtracting $m$ longitudes, giving $\mu + (n-m)\lambda$. It gives a surgery slope of $\mu +m\lambda$ on $R_2$ when $R_0 \neq R_2$. When $R_0 = R_2$, the slopes combine as on $R_1$ to give $\mu - (n-m)\lambda$. 
\end{proof}

\begin{definition}
Let $m$ be an even positive integer. Consider a unit cube fundamental region of $\TT^3$. For each $i = 1, 2, \ldots, \frac{m}{2}$, let $(R_{2i-1}^+, R_{2i-1}^-)$ be a pair of $(1,0,0)$-rods bounding a vertical $xz$-plane annulus within the unit cube, with $R_{2i-1}^+$ above and $R_{2i-1}^-$ below. Let $(R_{2i}^-, R_{2i}^+)$ be a pair of $(0,1,0)$-rods bounding a vertical $yz$-plane annulus with $R_{2i}^-$ above and $R_{2i}^+$ below. A rod $R$ is said to be \emph{sandwiched along the $xy$-plane by nested pairs of rods with order} $(R_1^{+}, R_2^{-}, \ldots, R_{m-1}^{+}, R_{m}^{-})$ if and only if $R$ lies in an $xy$-plane and the rods are positioned from top to bottom in the unit cube in the order \[
(R_1^{+}, R_2^{-},\ldots, R_{m-1}^{+}, R_{m}^{-}, R, R_{m}^{+}, R_{m-1}^{-}, \ldots, R_2^{+}, R_1^{-}).
\]
Similarly, for $m$ an odd positive integer, we can say $R$ is \emph{sandwiched along the $xy$-plane by nested pairs of rods with order} $(R_1^{+}, R_2^{-}, \ldots, R_{m-1}^{-}, R_{m}^{+})$ if and only if $R$ lies in an $xy$-plane and the rods are positioned from top to bottom in the unit cube in the order
\[
(R_1^{+}, R_2^{-},\ldots, R_{m-1}^{-}, R_{m}^{+}, R, R_{m}^{-}, R_{m-1}^{+}, \ldots, R_2^{+}, R_1^{-}).
\]
\end{definition}

See the top-right picture of \cref{Fig:Nested5On3} for an example of a rod (black) sandwiched by nested pairs of rods with $m = 2$.

\begin{lemma} \label{Lem:ToGetpq0Rod}
Let $p$ and $q$ be integers with $\gcd(p,q)=1$. Suppose that $[c_1;c_2, \ldots, c_m]$ is a continued fraction expansion of $p/q$. If $m$ is even, consider a $(1,0,0)$-rod $R_x$ sandwiched along the $xy$-plane by nested pairs of rods with order 
\[
(R_1^{+}, R_2^{-}, \ldots, R_{m-1}^{+}, R_{m}^{-}).
\] 
If $m$ is odd, consider a $(0,1,0)$-rod $R_y$ sandwiched along the $xy$-plane by nested pairs of rods with order
\[
(R_1^{+}, R_2^{-}, \ldots, R_{m-1}^{-}, R_{m}^{+}).
\]
Sequentially apply $(1/c_i)$-annular Dehn filling to the pair $(R_{i}^+, R_{i}^-)$ of rods, starting with $i = m$ and ending with $i = 1$. Then the rod $R_x$ for $m$ is even (respectively, $R_y$ for $m$ odd) is transformed to a $(p,q,0)$-rod. 
\end{lemma}

\begin{proof}
We will focus on the case when the length $m$ of the continued fraction is odd. The argument for $m$ even follows similarly.

Starting with the $(0,1,0)$-rod $R_y$ and applying $(1/c_m)$-annular Dehn filling to $(R_m^+, R_m^-)$ transforms the $(0,1,0)$-rod $R_y$ to a $(c_m, 1, 0)$-rod $R^{(1)}$. See the first and second pictures of \cref{Fig:Nested5On3} for an example.

The $(c_m, 1, 0)$-rod $R^{(1)}$ intersects the annulus bounded by $R_{m-1}^-$ and $R_{m-1}^+$ a total of $c_m$ times. Applying $(1/c_{m-1})$-annular Dehn filling to $(R_{m-1}^+, R_{m-1}^-)$ transforms the $(c_m, 1, 0)$-rod $R^{(1)}$ into a $(c_m, 1+c_m c_{m-1}, 0)$-rod $R^{(2)}$. See the second and third pictures of \cref{Fig:Nested5On3} for an example. Observe that the ratio of the rod parameters satisfies
\[
\frac{1+c_m c_{m-1}}{c_m} = c_{m-1} + \cfrac{1}{c_m}.
\]

The $(c_m, 1+c_m c_{m-1}, 0)$-rod $R^{(2)}$ intersects the annulus bounded by $R_{m-2}^+$ and $R_{m-2}^-$ a total of $1+c_m c_{m-1}$ times. Applying $(1/c_{m-2})$-annular Dehn filling to $(R_{m-2}^+, R_{m-2}^-)$ transforms the $(c_m, 1+c_m c_{m-1},0)$-rod $R^{(2)}$ into a $(c_m + (1+c_m c_{m-1})c_{m-2}, 1+c_m c_{m-1},0)$-rod $R^{(3)}$. See the third and fourth pictures of \cref{Fig:Nested5On3} for an example. Now observe that the ratio of the rod parameters satisfies
\[
\frac{c_m + (1+c_m c_{m-1})c_{m-2}}{1+c_m c_{m-1}} = c_{m-2} + \frac{c_m}{1+c_m c_{m-1}} = c_{m-2} + \cfrac{1}{c_{m-1} + \cfrac{1}{c_m}}
\]

Continuing in this way, we apply $(1/c_{m-3})$-annular Dehn filling, $(1/c_{m-4})$-annular Dehn filling, and so on, until we finally apply $(1/c_1)$-annular Dehn filling. Each successive annular Dehn filling prepends a term to the continued fraction expansion for the ratio of the rod parameters. Hence, the final rod $R^{(m)}$ has direction vector $(p,q,0)$, where $p/q = [c_1; c_2, \ldots, c_m]$. 
\end{proof}

Note that \cref{Lem:ToGetpq0Rod} holds for any continued fraction expansion of $p/q$, without any restriction on the signs of the terms.

\begin{definition} \label{Def:NestedAnnDehnFilling}
Let $p$ and $q$ be integers with $\gcd(p,q) = 1$. Suppose that $[c_1; c_2, \ldots, c_m]$ is a continued fraction expansion of $p/q$. Define {\em $(p,q)$-nested annular Dehn filling} to be the process of performing the sequence of $(1/c_i)$-annular Dehn fillings from $i = m$ to $i=1$ on the rod $R_x$ or $R_y$, as described in \cref{Lem:ToGetpq0Rod}. The rod $R_x$ or $R_y$ is called the \emph{core rod} of the nested annular Dehn filling. The rods $R_i^+$ and $R_i^-$ for $i = 1, 2, \ldots, m$ are called the \emph{filling rods} of the nested annular Dehn filling.
\end{definition}

For example, consider $(p,q)$-nested annular Dehn filling with $(p,q) = (5,3)$, using the continued fraction expansion $p/q = 5/3 = [1;1,2]$. Since the number of terms is odd, we start with a $(0,1,0)$-rod $R_y$ sandwiched along the $xy$-plane by nested pairs of rods with order $(R_1^+, R_2^-, R_3^+)$, as shown in the top-left picture of \cref{Fig:Nested5On3}. After applying $(1/2)$-annular Dehn filling to the pair of innermost red rods $(R_3^+, R_3^-)$, we obtain the rod complement shown in the top-right picture of \cref{Fig:Nested5On3}. Then after applying $(1/1)$-annular Dehn filling to the pair of green rods $(R_2^+, R_2^-)$, we obtain the rod complement shown in the bottom-left picture of \cref{Fig:Nested5On3}. Finally, after applying a $(1/1)$-annular Dehn filling to the outermost pair of red rods $(R_1^+, R_1^-)$, we obtain the rod complement shown in the bottom-right picture of \cref{Fig:Nested5On3}. The result is a single rod with direction vector $(5,3,0)$.

\begin{figure}
\includegraphics[scale=0.75]{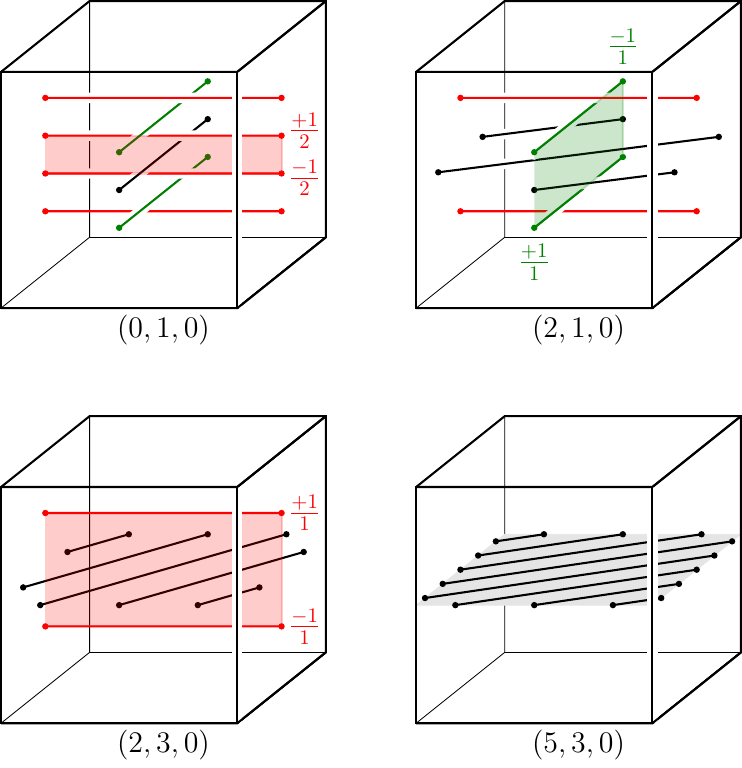}
\caption{The $(5,3)$-nested annular Dehn filling on the $(0,1,0)$-core rod. The vector under each $3$-torus is the direction vector of the corresponding black rod (up to ambient isotopy).}
\label{Fig:Nested5On3} 
\end{figure}

\begin{remark} 
Any rod that does not intersect the annulus used in annular Dehn filling is unaffected by the filling. In particular, such rods maintain their direction vectors. This straighforward observation is crucial for our use of annular Dehn fillings below.
\end{remark}

\section{Asymptotically sharp volume bounds} \label{Sec:3CuspedInfVol} 

With nested annular Dehn filling introduced in the last section, we can now proceed to show some asymptotically sharp volume bounds for a family of rod complements. 

\begin{lemma} \label{Lem:Hyperbolic}
Let $R_1, R_2, \ldots, R_n$ be disjoint rods in $\TT^3$ with $n \geq 3$. Suppose that $R_n$ has direction vector $(0, 0, 1)$ while each of the other rods $R_i$ has direction vector of the form $(p_i, q_i, 0)$. If any two neighbouring rods, ordered by $z$-coordinate, are not parallel, then the rod complement $\TT^3 \setminus (R_1 \cup R_2 \cup \cdots \cup R_n)$ is hyperbolic. 
\end{lemma}

\begin{proof}
The direction vectors of rods $R_1, R_2, R_n$ are linearly independent, since $R_1$ and $R_2$ are not parallel, and $R_n$ is orthogonal to the plane spanned by the direction vectors of $R_1$ and $R_2$. Since no two neighbouring rods are parallel, each pair of disjoint parallel rods are not linearly isotopic in the complement of the other rods. Thus, the result follows from \cref{Thm:ClassifyingT3Rods}.
\end{proof}

\begin{definition} \label{Def:StdParent}
A \emph{standard rod complement} is the complement of a finite number of rods in~$\TT^3$, each with direction vector $(1,0,0)$, $(0,1,0)$ or $(0,0,1)$.

A \emph{standard parent manifold} of a rod complement $\TT^3 \setminus (R_1 \cup R_2 \cup \cdots \cup R_n)$ is a standard rod complement from which $\TT^3 \setminus (R_1 \cup R_2 \cup \cdots \cup R_n)$ can be obtained after a finite sequence of Dehn fillings.
\end{definition}

\begin{proposition}[Standard parent manifolds exist] \label{Prop:StdParentMfd}
Let $R_1, R_2, \ldots, R_n$ be disjoint rods in $\TT^3$ with $n \geq 3$. Suppose that $R_n$ has direction vector $(0, 0, 1)$ while each of the other rods $R_i$ has direction vector of the form $(p_i, q_i, 0)$. Suppose that $p_i/q_i$ has a continued fraction expansion with $m_i$ terms. Let $E$ denote the number of $(p_i,q_i,0)$-rods with even $m_i$ and let $O$ denote the number of $(p_i,q_i,0)$-rods with odd $m_i$. Then there exists a standard rod complement $M$ with $E$ $(1,0,0)$-core rods and $O$ $(0,1,0)$-core rods together with $2\sum_{i=1}^{n-1} m_i$ filling rods such that $\TT^3 \setminus (R_1 \cup R_2 \cup \cdots \cup R_n)$ can be obtained by applying $(p_i, q_i)$-nested annular Dehn filling to the core rods of $M$ for $i = 1, 2, \ldots, n-1$. 
\end{proposition}

\begin{proof}
For $i = 1, 2, \ldots, n-1$, since $\gcd(p_i, q_i) = 1$, \cref{Lem:ToGetpq0Rod} and \cref{Def:NestedAnnDehnFilling} ensure that the $(p_i, q_i, 0)$-rod $R_i$ can be obtained by applying a $(p_i, q_i)$-nested annular Dehn filling to one of the $E+O$ core rods. The $2m_i$ filling rods sandwiching the core rod will be removed in the process of Dehn filling. Observe that a $(p_i, q_i)$-nested annular Dehn filling does not affect the isotopy classes of rods disjoint from the associated annuli. Hence, after applying $n-1$ nested annular Dehn fillings on the $E+O = n-1$ core rods, we obtain a $3$-manifold homeomorphic to $\TT^3 \setminus (R_1 \cup R_2 \cup \cdots \cup R_n)$.
\end{proof}

\cref{Prop:StdParentMfd} provides an explicit procedure to obtain a standard parent manifold of a rod complement with the particular form for which the result applies. The manifold $M$ in \cref{Prop:StdParentMfd} is a standard parent manifold of $\TT^3 \setminus (R_1 \cup R_2 \cup \cdots \cup R_n)$. Note that for each sandwich of a nested annular Dehn filling, the outermost pair of filling rods are $(1,0,0)$-rods. Between each pair of adjacent (possibly the same) sandwiches, the bottom filling rod of the top sandwich is linearly isotopic to the top filling rod of the bottom sandwich, so there is a natural choice of essential plane annulus between these two filling rods.  To obtain a hyperbolic standard parent manifold, we cut along any such essential plane annuli in $M$. Observe this merges two parallel rods into a single rod.

An example of a standard parent manifold with essential annuli and two core rods is shown in \cref{Fig:Standard}. Note that \cref{Fig:Parent_p7q7} below shows a hyperbolic standard parent manifold. For that example, black and pale green rods are both core rods, and the outermost red rods correspond to the top and bottom filling rods of the sandwiches. 

\begin{figure}
    \includegraphics[width=2in]{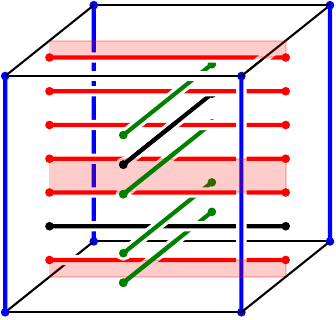}
    \caption{An example of a standard parent manifold with essential annuli.}
    \label{Fig:Standard}
\end{figure}


\begin{lemma} \label{Lem:OctahedralDecomposition}
Consider a standard parent manifold $M$ with exactly one $(0,0,1)$-rod and $m\geq 2$ additional rods, which alternate between $(1,0,0)$-rods and $(0,1,0)$-rods. Then $M$ is hyperbolic and can be decomposed into $m$ regular ideal octahedra. Thus, its volume is $\vol(M) = m \, \voct$, where $\voct \approx 3.66386$ is the volume of the regular ideal octahedron. 
\end{lemma}

\begin{proof}
The fact that $M$ is hyperbolic follows from \cref{Thm:ClassifyingT3Rods}. Alternatively, one can construct the hyperbolic structure directly as follows. Cut $M$ along an $xz$-plane torus, a $yz$-plane torus, and all $xy$-plane tori that contain $(1,0,0)$-rods or $(0,1,0)$-rods. We obtain $m$ three-dimensional balls, each with six arcs removed from the boundary. By shrinking these arcs, one obtains $m$ ideal octahedra; see \cref{Fig:OctahedraMidpoints}. 

We can assign a complete hyperbolic metric on $M$ by setting each ideal octahedron to be regular. Such a polyhedron has dihedral angles equal to $\pi/2$. The gluing of the octahedra identifies four such dihedral angles around each edge and tiles each cusp by Euclidean squares, so one obtains a complete hyperbolic structure; see~\cite[Theorem~4.10]{Purcell:HyperbolicKnotTheory}. The volume of $M$ is then $m \, \voct$, the sum of the volumes of the octahedra.
\end{proof}

\begin{lemma} \label{Lem:CuspShapeAndSize}
Let $M$ be a hyperbolic standard parent manifold. The fundamental region of the torus cusp boundary corresponding to each filling rod of $M$ is a Euclidean rectangle formed by gluing two squares corresponding to cusp neighbourhoods of ideal vertices of octahedra. The meridian forms one of the sides of the rectangle, running along one edge of each square. The longitude forms the other side of the rectangle, running along an edge of one of the squares. Finally, there exists a choice of horoball neighbourhoods with disjoint interiors for the rod complement such that the meridian has length 2 and the longitude has length 1. 
\end{lemma}

\begin{proof}
Consider how the octahedra in the proof of \cref{Lem:OctahedralDecomposition} fit together. Since the rod complement can be decomposed into ideal octahedra, the cusps corresponding to the filling rods are tiled by Euclidean squares that are cusp neighbourhoods of the ideal vertices of the octahedra.

Note that each horizontal rod $R$ meets exactly two octahedra: one above the $xy$-plane containing $R$, which we cut along to obtain the decomposition, and one below. The meridian~$\mu$ runs once through each and can be isotoped to run through the $xz$- or $yz$-plane as in the left of \cref{Fig:OctahedraMidpoints}. Hence, it lies on faces of the two octahedra. Thus, the meridian forms a closed curve running along one edge in each of the two squares corresponding to the two octahedra.

\begin{figure}
\includegraphics[scale=0.75]{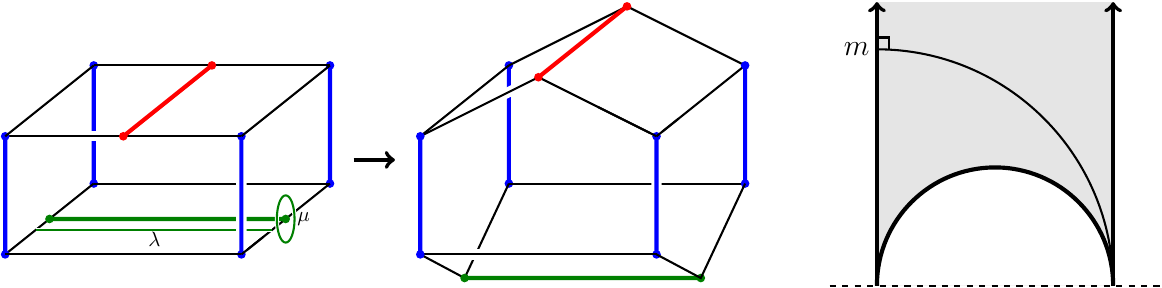}
\caption{Left: The complement of rods in a standard parent manifold can be decomposed into ideal octahedra. Shown are the meridian $\mu$ and the longitude $\lambda$ for the green rod on the bottom. Right: The midpoint $m$ of an ideal edge in a hyperbolic ideal triangle.}
\label{Fig:OctahedraMidpoints}
\end{figure}

The longitude may be isotoped to run through a single octahedron, say the one above the $xy$-plane containing $R$, as in the left of \cref{Fig:OctahedraMidpoints}. Thus, it forms one side of a cusp square. Finally, observe that the square is glued to itself by the identity, with one side glued to the opposite side. 
The cusp is a Euclidean rectangle, comprised of two squares, with the meridian running along the long edge of the rectangle and the longitude running along the short edge.

It remains to argue that the lengths of the meridian and the longitude are 2 and 1, respectively. To do so, we show that we can choose horoballs about the cusps of $M$ with disjoint interiors such that when we intersect with the ideal octahedra, the boundary of the intersection is a collection of squares, each with side length 1. The horoball expansion we use is the same as that appearing in \cite[Lemma~7.22]{Purcell:HyperbolicKnotTheory} or \cite[Lemma~3.7]{FuterPurcell:LinksNoExceptionalSurgeries}. 
That is, each edge $e$ of the octahedron borders two triangular faces. The midpoint of the edge $e$ with respect to one of the triangles is the unique point on the edge $e$ that lies on a perpendicular hyperbolic geodesic running from the opposite vertex to $e$; see the right of \cref{Fig:OctahedraMidpoints}. Since our ideal octahedron is regular, the midpoints obtained from either adjacent triangle agree. When the vertices of the ideal triangle are placed at 0, 1 and $\infty$, the midpoint has height 1. If we place a regular ideal octahedron containing a side with vertices at 0, 1, and $\infty$, the midpoints of each of the edges meeting infinity also have height 1. This remains true after applying a M\"{o}bius transformation taking any vertex to infinity. Thus, we may expand horoballs about each ideal vertex to the height of the midpoints of the four edges meeting that vertex. This gives a collection of horoballs that are tangent exactly at the midpoints of edges, with disjoint interiors. The boundary of each horoball meets the octahedron in a square of side length $1$. Finally, since the octahedra are glued in such a way that cusp squares glue to cusp squares with the same side lengths, this gluing must preserve this choice of horoballs. Hence, these define horoball neighbourhoods with disjoint interiors and lengths as claimed. 
\end{proof}

\begin{lemma} \label{Lem:SlopeLength}
Let $M$ be a hyperbolic standard parent manifold, with slope $1/n$ on one of the horizontal rods. Then in the horoball neighbourhood described in \cref{Lem:CuspShapeAndSize}, the length of the slope is $\sqrt{n^2+4}$. 
\end{lemma}

\begin{proof}
The slope $1/n$ runs once along a meridian and $n$ times along the longitude. In the universal cover of the cusp torus, it can be lifted to an arc with one endpoint at $(0,0)$ and the other at $(2,n)$. The meridian and longitude are orthogonal, with the meridian of length $2$ and the longitude of length $1$. Hence, length of the slope is $\sqrt{n^2+2^2}$. 
\end{proof}

We are now ready to prove the coarse volume bound discussed in the introduction.

\begin{theorem} \label{Thm:MainOrthogonalRods}
\Paste{MainOrthogonalRods}
\end{theorem}

\begin{proof}
By \cref{Lem:Hyperbolic}, the manifold $M$ must be hyperbolic.

We construct standard parent manifolds with ideal octahedral decompositions. By \cref{Prop:StdParentMfd}, there exists a standard rod complement $N$ with $n-1$ core rods and $\sum_{i=1}^{n-1} 2m_i$ filling rods such that $M$ can be obtained by applying a $(p_i,q_i)$-nested annular Dehn filling to each of the core rods of $N$. Observe that the outermost pair of filling rods for each nested annular Dehn filling are $(1,0,0)$-rods. Each of the two outermost filling rods for each nested annular Dehn filling will be linearly isotopic to an outermost filling rod for another nested annular Dehn filling. By cutting along the essential annuli arising from all of these linear isotopies, we obtain a standard parent manifold $N_\textup{p}$ with exactly one $(0,0,1)$-rod, namely $R_n$, and alternating $(1,0,0)$-rods and $(0,1,0)$-rods.

By \cref{Lem:OctahedralDecomposition}, $N_\textup{p}$ has a decomposition into $\sum_{i=1}^{n-1} 2m_i$ regular ideal octahedra and it admits a complete hyperbolic structure.

We obtain $M=\TT^3 \setminus (R_1 \cup R_2 \cup \cdots \cup R_n)$ by Dehn filling the standard parent manifold $N_\textup{p}$. Since Dehn filling decreases volume~\cite{Thurston:Geom&TopOf3Mfd}, we obtain the bound
\[
\vol(M) < \vol(N_\textup{p}) = \voct\,\sum_{i=1}^{n-1} 2m_i.
\]
Furthermore, this bound is asymptotically sharp. Taking larger and larger values for the coefficients $c_{ij}$ of the continued fraction expansion while fixing the lengths $m_i$ will produce Dehn fillings of the same parent manifold whose volumes converge to that of the parent manifold.

For the lower bound, we consider the slopes of the Dehn filling. These are of the form $1/c_{ij}$ for filling components with $2 \leq j \leq m_i$. For the outermost filling rods, the coefficient of the Dehn filling combines the $1/c_{i1}$ from one side with $-1/c_{(i-1)1}$ from the other side, as in \cref{Lem:AnnularDFIndependentA}. Thus, the slope is $1/(c_{i1}-c_{(i-1)1})$.

By \cref{Lem:CuspShapeAndSize}, for any integer $\ell$, the length of the slope $1/\ell$ on a filling rod is $\sqrt{\ell^2 + 4}$. So under the hypotheses required for the lower bound, the minimum length slope will be at least $\sqrt{6^2+4} > 2\pi$. We may now apply a theorem of Futer, Kalfagianni and Purcell, which states that if the minimum slope length is larger than $2\pi$, then the volume change under Dehn filling is a multiple of the volume of the unfilled manifold~\cite[Theorem~1.1]{Futer-Kalfagianni-Purcell:DehnFillingVolumeJonesPolynomial}. In our case, this leads to
\[
\vol(M) \geq \left( 1 - \frac{4\pi^2}{C^2+4} \right)^{3/2} 2 \, \voct \sum_{i=1}^{n-1} m_i. \qedhere
\]
\end{proof}

\begin{remark}
The upper bound of \cref{Thm:MainOrthogonalRods} motivates one to seek an efficient expression for such rod complements, with the complexity measured by $\sum_{i=1}^{n-1} m_i$, the sum of the lengths of the continued fractions. One may simultaneously switch each $(p_i,q_i,0)$-rod to a $(q_i,p_i,0)$-rod, which may change $\sum_{i=1}^{n-1} m_i$. Recall that we allow negative terms in our continued fractions, as per the discussion in \cref{subsec:continued}. Typically, one obtains shorter continued fractions this way than if one restricts to using positive integers as terms.
\end{remark}

\begin{corollary} \label{Cor:BadUpperBound}
\Paste{BadUpperBound}
\end{corollary}

\begin{proof}
For $n$ a positive integer, let $R_1^{(n)}$ be an $(n,1,0)$-rod, let $R_2$ be a $(0,1,0)$-rod, and let $R_3$ be a $(0,0,1)$-rod. These rods satisfy the hypotheses of the first part of \cref{Thm:MainOrthogonalRods}. Note that the continued fraction associated to the rod $R_1^{(n)}$ is $n/1 = [n]$. Thus, in the notation of \cref{Thm:MainOrthogonalRods}, we have $m_1 = 1$ for any choice of $n$ and we also have $m_2 = 1$. So the upper bound of \cref{Thm:MainOrthogonalRods} implies that 
\[
\vol \left( \TT^3 \setminus (R_1^{(n)} \cup R_2 \cup R_3) \right) \leq 4 \, \voct.
\]
On the other hand, we have $(p_1, q_1) = (n, 1)$ and $(p_2, q_2) = (0, 1)$, so 
$|p_1q_2 - p_2q_1| = n$, which is unbounded as $n$ grows to infinity.
\end{proof}

\begin{corollary} \label{Cor:3CuspedInfVol}
\Paste{3CuspedInfVol}
\end{corollary}

\begin{proof} 
Define the sequence of rational slopes
\[
p_k/q_k = [\underbrace{k; k, k, \ldots, k}_{k \text{ terms}}].
\] 
for $k \geq 6$. For example, we have
\begin{align*}
p_6/q_6 &= [6; 6, 6, 6, 6, 6] = 53353/8658, \\ 
p_7/q_7 &= [7; 7, 7, 7, 7, 7, 7] = 927843/129949, \\ 
p_8/q_8 &= [8; 8, 8, 8, 8, 8, 8, 8] = 18674305/2298912. 
\end{align*}

Let $R_1^{(k)}$ be a $(p_k, q_k, 0)$-rod, let $R_2$ be a $(0,1,0)$-rod, and let $R_3$ be a $(0,0,1)$-rod. Let $M_k = \TT^3 \setminus (R_1^{(k)} \cup R_2 \cup R_3)$ be the associated rod complement. Using the notation of \cref{Thm:MainOrthogonalRods}, we have $m_1 = k$, $m_2 = 1$, and $C = k \geq 6$. \Cref{Fig:Parent_p7q7} shows the standard parent manifold of $\TT^3 \setminus(R_1^{(7)}\cup R_2\cup R_3)$. So \cref{Thm:MainOrthogonalRods} implies that
\[
\vol(M_k) \geq \left( 1 - \frac{4\pi^2}{k^2+4} \right)^{3/2} 2 \, \voct \, (k+1) > \left( 1 - \frac{4\pi^2}{6^2+4} \right)^{3/2} 2 \, \voct \, k > 0.01091 k.
\]
Since the right side grows to infinity with $k$, the volume of $M_k$ also grows to infinity. 
\end{proof}

\begin{figure}[ht!]
    \includegraphics[width=2in]{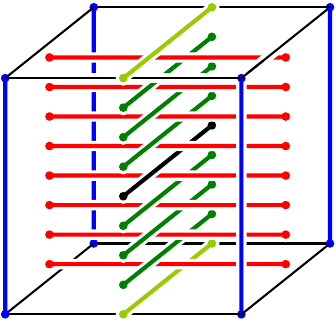}
\caption{The standard parent manifold of $\TT^3 \setminus(R_1^{(7)}\cup R_2\cup R_3)$. The black rod is the core rod with direction vector $(0,1,0)$; the pale green $(0,1,0)$-rod that lies on the boundary (top and bottom) of the unit cube is $R_2$; the blue $(0,0,1)$-rod is $R_3$.}
\label{Fig:Parent_p7q7} 
\end{figure}

\section{Further discussion} \label{Sec:Discussion} 

Our results on the volumes of rod complements suggest various natural questions worthy of further exploration, such as the following.

\begin{question} \label{Question:Coarse}
Do there exist two-sided coarse volume bounds for all rod complements in terms of the rod parameters?
\end{question}

By \cref{Cor:BadUpperBound,Cor:3CuspedInfVol}, such bounds cannot depend only on the number of rods nor on the number of intersections of the rods in a particular projection. It would be natural to wonder whether two rod complements with the same rod parameters have volumes with bounded ratio.

\begin{question} \label{Question:Mutation}
Does hyperbolic volume distinguish rod complements up to homeomorphism?
\end{question}

It would be surprising if any two rod complements with the same hyperbolic volume were necessarily homeomorphic. It is well-known that hyperbolic volume does not distinguish hyperbolic 3-manifolds in general. In particular, mutation of cusped hyperbolic 3-manifolds can change its homeomorphism class, but necessarily preserves the hyperbolicity and volume~\cite{Ruberman:Mutation}. An example of mutation involves cutting along an essential embedded 4-punctured sphere bounding a tangle in a ball, rotating the ball via a certain involution, and then regluing. Mutation can also be performed with respect to surfaces of other topologies that possess a suitable involution. It is not immediately obvious whether rod complements contain such embedded essential surfaces along which mutation can be performed.

\begin{question}
Does there exist a rod complement with a non-trivial mutation?
\end{question}

\bibliographystyle{amsplain}
\bibliography{rod-complement-volume-bounds}

\end{document}